\documentclass[12pt]{amsart}

\title{The structure of $\Rlimit$-groups}
\author{Ioannis Emmanouil}\thanks{Research funded by the Hellenic
Foundation for Research and Innovation (H.F.R.I.) under the "3rd
Call for H.F.R.I.\ Research Projects to Support Faculty Members
and Researchers", project number 24921}

\newtheorem{Lemma}{Lemma}[section]
\newtheorem{Proposition}[Lemma]{Proposition}
\newtheorem{Theorem}[Lemma]{Theorem}
\newtheorem{Corollary}[Lemma]{Corollary}

\newcommand{\Limit}{\mbox{$\displaystyle{\lim_{\longleftarrow}}$}}
\newcommand{\Limitn}{\mbox{$\displaystyle{\lim}_n\!\!$}}

\newcommand{\Rlimit}{\mbox{$\displaystyle{\lim}^1$}}

\newcommand{\Rlimitn}{\mbox{$\displaystyle{\lim}_n^1\!$}}

\begin{document}

\begin{abstract}
If $(A_n)_n$ is a decreasing filtration of a module $A$
and $\widehat{A} = \Limitn\ A/A_n$, then $\Rlimitn\ A_n$
is identified with the cokernel of the canonical map
$A \longrightarrow \widehat{A}$. In this note, we show
that any $\Rlimit$-group is canonically of that form:
For any inverse sequence of modules $(X_n)_n$ there exists
an inverse sequence $(A_n)_n$ as above and a morphism
$(A_n)_n \longrightarrow (X_n)_n$, depending functorially
on $(X_n)_n$, that induces an isomorphism on $\Rlimit$.
The proof is based on Quillen's small object argument,
as formulated by Eklof and Trlifaj in their investigation
of the existence of enough injective objects in certain
cotorsion pairs, and also uses a construction by Salce that
provides enough projective objects therein.
\end{abstract}

\maketitle

\addtocounter{section}{-1}
\section{Introduction}

\noindent
The right derived functors of the limit functor in module categories were
introduced by Milnor \cite{M} and Yeh \cite{Y}, in order to compute such
topological invariants as the homotopy groups of the limit of a sequence of 
fibrations or the cohomology of an ascending union of spaces. Roos showed 
in \cite{R} that the higher limits vanish in degrees $\geq 2$, if the 
indexing category is the ordered set of natural numbers. (This assumption 
will hold throughout this note.)

The most manageable criterion that implies the vanishing of $\Rlimit$ is the
Mittag-Leffler condition, that was introduced by Grothendieck in \cite{Gro}.
This condition turns out to be equivalent to the vanishing of $\Rlimit$ for
inverse sequences of countable abelian groups \cite{Gra}, but the relation
between the two conditions is slightly more complicated in general; cf.\
\cite[Theorem 1.3]{ABH} or \cite[Corollary 6]{E}. It is shown in \cite{CI}
that the vanishing of $\Rlimit$ is equivalent to a certain transfinite
version of the Mittag-Leffler condition. 

Another class of examples over which the computation of $\Rlimit$ is 
well-understood is provided by decreasing filtrations: If $(A_n)_n$ 
is a decreasing sequence of submodules of a module $A=A_0$ and 
$\widehat{A} = \Limitn\ A/A_n$ is the Hausdorff completion of $A$ with 
respect to the linear topology that is induced by the filtration, then 
$\Rlimitn\ A_n$ is identified with the cokernel of the canonical map 
$A \longrightarrow \widehat{A}$. Therefore, assuming that the linear 
topology on $A$ is already Hausdorff, the group $\Rlimitn\ A_n$ 
measures the obstruction to completeness of $A$. The reader may consult 
\cite[$\S $3.5]{W} for a proof of this result. Our goal in this paper 
is to show that any $\Rlimit$-group is of that form, i.e.\ it coincides 
with the value of $\Rlimit$ on an inverse sequence that is associated 
with a suitable decreasing filtration. More precisely, we prove the 
following result.

\medskip

\noindent
{\bf Theorem.}
{\em For any inverse sequence of modules $(X_n)_n$ there is an another 
inverse sequence of modules $(A_n)_n$ and a morphism of inverse sequences
$f : (A_n)_n \longrightarrow (X_n)_n$, which depends functorially on 
$(X_n)_n$, such that:

(i) the inverse sequence $(A_n)_n$ represents a decreasing filtration of 
$A=A_0$ by submodules,

(ii) the modules $A_n$ are free for all $n$,

(iii) the successive quotients $A_n/A_{n+1}$ are free modules for all $n$ 
and 

(iv) $f$ induces an isomorphism 
$\Rlimitn\ A_n \stackrel{\sim}{\longrightarrow} \Rlimitn\ X_n$.}

\medskip

\noindent
The proof proceeds by working in the abelian category ${\mathcal C}$ of 
inverse sequences of modules. We express $\Rlimit$ as the Ext$^1$-functor
on ${\mathcal C}$ that is associated with a certain inverse sequence and 
consider the cotorsion pair in ${\mathcal C}$ that is cogenerated by that 
inverse sequence. The reader is referred to \cite{ET} for the definition
and basic properties of cotorsion pairs. We use the version of Quillen's 
small object that was formulated by Eklof and Trlifaj \cite{ET, T} and a
construction by Salce \cite{S} to obtain the completeness of that cotorsion 
pair. It will turn out that the result stated above is an immediate consequence 
of this completeness.

\vspace{0.1in}

\noindent
{\em Notations and terminology.} We fix a unital associative ring 
$R$ and consider only left modules over $R$. As mentioned above, 
with the proof of the Theorem in mind, we only consider inverse 
sequences of modules in this paper and refrain from considering 
diagrams of modules indexed by more general directed sets or other 
small categories.

\section{An Ext-description of $\Rlimit$}

\noindent
We consider the abelian category ${\mathcal C}$ of inverse sequences
of modules; it is precisely the category of contravariant functors
from the ordered set $\omega$ to the category of modules. In this
section, we elaborate on the discussion in \cite[$\S $2]{CI} and
identify the $\Rlimit$-functor to the Ext$^1$-functor on
${\mathcal C}$ induced by a suitable inverse sequence, a result
that is certainly known to the experts.

Let $X=(X_n)_n$ be an inverse sequence of modules with structure 
maps $(s_n : X_n \longrightarrow X_{n-1})_n$. We consider the 
linear map 
\[ 1 - S : {\displaystyle{\prod}}_nX_n \longrightarrow 
   {\displaystyle{\prod}}_nX_n , \]
whose composition with the projection onto $X_n$ coincides with 
the projection from the product to $X_n \oplus X_{n+1}$, followed 
by $[1,-s_{n+1}] : X_n \oplus X_{n+1} \longrightarrow X_n$, for all
$n$. Then, $\lim_n X_n = \mbox{ker}(1-S)$ and 
$\Rlimitn\ X_n = \mbox{coker}(1-S)$. The reader is referred to 
Weibel's book \cite[$\S $3.5]{W} for background and general 
properties of the $\Rlimit$-functor.

If $M$ is a module and $n$ a non-negative integer, then we may
consider the inverse sequence $M(n)$, consisting of $M$ in degrees
$\leq n$ and $0$ in degrees $>n$, with all structure maps
$M \longrightarrow M$ given by the identity map of $M$. It is
easily seen that for any inverse sequence $X=(X_i)_i$ there is 
an isomorphism of abelian groups
\[ \phi_n : \mbox{Hom}_{\mathcal C}(M(n),X)
   \stackrel{\sim}{\longrightarrow} \mbox{Hom}_R(M,X_n) , \]
which is natural in both $M$ and $X$. Since exactness in
${\mathcal C}$ is defined degreewise, we conclude that $M(n)$ is
a projective object in ${\mathcal C}$ for all $n \geq 0$, if $M$
is a projective module. We also note that there is an obvious
monomorphism $\iota_n : M(n) \longrightarrow M(n+1)$ in
${\mathcal C}$ for all $n \geq 0$, which is natural in $M$; it
is given by the identity map of $M$ in degrees $\leq n$. Moreover,
the following diagram is commutative for any inverse sequence 
$X=(X_i)_i$
\begin{equation}
\begin{array}{ccc}
 \mbox{Hom}_{\mathcal C}(M(n+1),X) &
 \stackrel{\phi_{n+1}}{\longrightarrow} &
 \mbox{Hom}_R(M,X_{n+1}) \\
 {\scriptstyle{\imath_n^*}} \downarrow & &
 \downarrow {\scriptstyle{(s_{n+1})_*}} \\
 \mbox{Hom}_{\mathcal C}(M(n),X) &
 \stackrel{\phi_{n}}{\longrightarrow} &
 \mbox{Hom}_R(M,X_n)
 \end{array}
\end{equation}
Here, we denote by $s_{n+1} : X_{n+1} \longrightarrow X_n$ 
the structure map of $X$; the additive map $(s_{n+1})_*$ is
the corresponding structure map of the inverse sequence 
$(\mbox{Hom}_R(M,X_i))_i$. This diagram is natural in 
both $M$ and $(X_i)_i$. The colimit $M(\infty) = \mbox{colim}_nM(n)$
is the inverse sequence consisting of $M$ in each degree, with all
structure maps given by the identity of $M$. There is a short exact
sequence in ${\mathcal C}$
\begin{equation}
 0 \longrightarrow {\textstyle{\bigoplus_n}} M(n)
   \stackrel{1-I}{\longrightarrow} {\textstyle{\bigoplus_n}} M(n)
   \longrightarrow M(\infty) \longrightarrow 0 ,
\end{equation}
which is natural in $M$. Here, $1-I$ is defined so that its restriction
to $M(n)$ is the composition of
$(1,-\iota_n) : M (n) \longrightarrow M(n) \oplus M(n+1)$, followed
by the inclusion of $M(n) \oplus M(n+1)$ into the direct sum, for all 
$n$.

\begin{Proposition}
Let $M$ be a module and $X=(X_i)_i$ and inverse sequence of modules.

(i) $\mbox{Hom}_{\mathcal C}(M(\infty),X) \simeq
     \lim_i \mbox{Hom}_R(M,X_i)$

(ii) If $M$ is a projective module, then
     $\mbox{Ext}^1_{\mathcal C}(M(\infty),X) \simeq
     \lim_i^1 \mbox{Hom}_R(M,X_i)$.
\newline
The isomorphisms in (i) and (ii) above are natural in both $M$
and $X$; in particular, these are isomorphisms of right
$\mbox{End}_RM$-modules.
\end{Proposition}
\vspace{-0.05in}
\noindent
{\em Proof.}
(i) The short exact sequence (2) identifies
$\mbox{Hom}_{\mathcal C}(M(\infty),X)$ with the kernel of
\[ (1-I)^* : \mbox{Hom}_{\mathcal C} \! \left(
   \textstyle{\bigoplus_n} M(n),X \right)  \longrightarrow
   \mbox{Hom}_{\mathcal C} \! \left(
   \textstyle{\bigoplus_n} M(n),X \right) , \]
i.e.\ with the kernel of
\[ 1-I^* : {\displaystyle{\prod}}_n
   \mbox{Hom}_{\mathcal C} (M(n),X) \longrightarrow
   {\displaystyle{\prod}}_n
   \mbox{Hom}_{\mathcal C} (M(n),X) . \]
In view of the commutative diagram (1), the latter map is
identified with
\[ 1-S_* : {\displaystyle{\prod}}_n  \mbox{Hom}_R(M,X_n)
   \longrightarrow {\displaystyle{\prod}}_n
   \mbox{Hom}_R(M,X_n) \]
and hence
$\mbox{Hom}_{\mathcal C}(M(\infty),X) =
 \mbox{ker} (1-S_*) = \lim_i \mbox{Hom}_R(M,X_i)$.
Naturality follows, since (1) and (2) are natural in both
$M$ and $X$.

(ii) If $M$ is projective, then (2) is a projective resolution
of $M(\infty)$ in ${\mathcal C}$ and hence we may compute
$\mbox{Ext}^1_{\mathcal C}(M(\infty),X)$ as the cokernel 
of the maps displayed in (i) above. It follows that
$\mbox{Ext}^1_{\mathcal C}(M(\infty),X) =
 \mbox{coker} (1-S_*) = \lim_i^1 \mbox{Hom}_R(M,X_i)$,
naturally in both $M$ and $X$. \hfill $\Box$

\vspace{0.1in}

\noindent
We now specialize the discussion above to the case where
$M=R$ is the regular module and note that the ring of
endomorphisms of the regular module $R$ is the opposite 
ring $R^{op}$.

\begin{Corollary}
For any inverse sequence of modules $X=(X_n)_n$ there are
natural isomorphisms of (left $R$-)modules
$\mbox{Hom}_{\mathcal C}(R(\infty),X) \simeq \lim_nX_n$
and
$\mbox{Ext}^1_{\mathcal C}(R(\infty),X) \simeq \Rlimitn\ X_n$.
\hfill $\Box$
\end{Corollary}

\section{The existence of special preenvelopes}

\noindent
In view of Corollary 1.2, the roots of $\Rlimit$ in
${\mathcal C}$ are precisely the roots of
$\mbox{Ext}^1_{\mathcal C}(R(\infty),\_\!\_)$. The
fundamental work by Eklof and Trlifaj \cite{ET} can
be easily extended from a module category to an abelian
category such as ${\mathcal C}$. In this section, we
shall use their argument (adapted in our setting) and
construct preenvelopes in ${\mathcal C}$ by the roots
of $\Rlimit$.

If ${\mathcal C}_0 \subseteq {\mathcal C}$ is a class of
inverse sequences, then an inverse sequence $X$ is said
to be ${\mathcal C}_0$-filtered if there exists an ordinal
number $\tau$ and an ascending family of inverse subsequences
$(X^{\alpha})_{\alpha \leq \tau}$ of $X$, such that:

(i) $X^0=0$ and $X^{\tau} = X$,

(ii) for any limit ordinal $\alpha \leq \tau$ we have
$X^{\alpha} = \bigcup_{\beta < \alpha} X^{\beta}$ and

(iii) for any $\alpha < \tau$ the quotient
$X^{\alpha +1}/X^{\alpha}$ is isomorphic to an inverse
sequence in ${\mathcal C}_0$.
\newline
Let $\mbox{Filt}({\mathcal C}_0)$ be the class of
${\mathcal C}_0$-filtered inverse sequences. We note that
$\mbox{Filt}({\mathcal C}_0)$ is closed under isomorphisms,
contains all direct sums of inverse sequences in
${\mathcal C}_0$ and
$\mbox{Filt}(\mbox{Filt}({\mathcal C}_0)) =
 \mbox{Filt}({\mathcal C}_0)$.
If ${\mathcal C}_0 = \{ C \}$ is a singleton, we denote
$\mbox{Filt}({\mathcal C}_0)$ by $\mbox{Filt}(C)$.

\begin{Lemma}
The class $\mbox{Filt}(R(\infty))$ consists of all inverse
sequences isomorphic to $F(\infty)$ for a suitable free
module $F$.
\end{Lemma}
\vspace{-0.05in}
\noindent
{\em Proof.}
If $F$ is free with basis $\Lambda$, then $F(\infty)$ is
the direct sum of $\Lambda$ copies of $R(\infty)$ and hence
$F(\infty) \in \mbox{Filt}(R(\infty))$. Conversely, assume
that $X = (X_n)_n \in \mbox{Filt}(R(\infty))$ and let
$(X^{\alpha})_{\alpha \leq \tau}$ be an ascending family
of inverse subsequences of $X$, as in the definition. Then,
for each $n$, the module $X_n$ is filtered by $R$ and hence
it is free. Moreover, we may use transfinite induction to
show that the structure maps of the inverse sequence
$X^{\alpha}$ are bijective for all $\alpha \leq \tau$. In
particular, the structure maps of $X = X_{\tau}$ are
bijective and hence $X \simeq X_0(\infty)$. \hfill $\Box$

\vspace{0.1in}

\noindent
The following result is essentially \cite[Theorem 2]{ET},
adapted in our setting; see also \cite[Theorem 2.2]{T}.
We provide the details of the proof for the reader's
convenience.

\begin{Theorem}
For any inverse sequence $X=(X_n)_n$ there is a functorial
short exact sequence of inverse sequences
\[ 0 \longrightarrow (X_n)_n \longrightarrow (Y_n)_n
     \longrightarrow (Z_n)_n \longrightarrow 0 , \]
where $\Rlimitn\ Y_n=0$ and $(Z_n)_n \simeq F(\infty)$ for
a free module $F$.
\end{Theorem}
\vspace{-0.05in}
\noindent
{\em Proof.}
We consider the projective resolution
\begin{equation}
 0 \longrightarrow {\textstyle{\bigoplus_n}} R(n)
   \stackrel{1-I}{\longrightarrow}
   {\textstyle{\bigoplus_n}} R(n)
   \longrightarrow R(\infty) \longrightarrow 0
\end{equation}
of $R(\infty)$ (cf.\ (2)) and let $\tau = \aleph_1$. We
shall construct an ascending union of inverse sequences
$(Y^{\alpha})_{\alpha \leq \tau}$, such that:

(i) $Y^0=X$,

(ii) for any limit ordinal $\alpha \leq \tau$ we have
$Y^{\alpha} = \bigcup_{\beta < \alpha} Y^{\beta}$ and

(iii) for any $\alpha < \tau$ the quotient
$Y^{\alpha +1}/Y^{\alpha}$ is isomorphic to a direct sum
of copies of $R(\infty)$ and for any morphism
$f : \bigoplus_nR(n) \longrightarrow Y^{\alpha}$ there
is a morphism
$g : \bigoplus_nR(n) \longrightarrow Y^{\alpha +1}$
that fits into a commutative diagram
\begin{equation}
\begin{array}{ccc}
   {\textstyle{\bigoplus_n}} R(n) &
   \stackrel{1-I}{\longrightarrow} &
   {\textstyle{\bigoplus_n}} R(n) \\
   \!\!\! {\scriptstyle{f}} \downarrow & &
   \downarrow {\scriptstyle{g}} \\
   Y^{\alpha} & \stackrel{\iota_{\alpha}}{\longrightarrow} &
   Y^{\alpha +1}
\end{array}
\end{equation}
where $\iota_{\alpha}$ denotes the inclusion.

We define the $Y^{\alpha}$'s by transfinite induction. Of 
course, we start with $Y^0=X$ and proceed as dictated by 
(ii) for limit ordinals. Assuming that $\alpha < \tau$ and 
that we have already constructed the inverse sequence 
$Y^{\alpha}$, we consider the set $\Lambda^{\alpha}$ of 
all morphisms 
$f : \bigoplus_nR(n) \longrightarrow Y^{\alpha}$ and the
short exact sequence
\[ 0 \longrightarrow \! \left[ {\textstyle{\bigoplus_n}}
     R(n) \right] \! ^{(\Lambda^{\alpha})}
     \stackrel{J}{\longrightarrow}
     \! \left[ {\textstyle{\bigoplus_n}}
     R(n) \right] \! ^{(\Lambda^{\alpha})}
     \longrightarrow R(\infty) ^{(\Lambda^{\alpha})}
     \longrightarrow 0 , \]
which is obtained as the direct sum of $\Lambda^{\alpha}$
copies of (3). Here, $J = (1-I)^{(\Lambda^{\alpha})}$ and
hence we have $J \circ \nu_f = \nu_f \circ (1-I)$, where
$\nu_f : \bigoplus_nR(n) \longrightarrow \left[
 {\textstyle{\bigoplus_n}} R(n) \right]
 \! ^{(\Lambda^{\alpha})}$
denotes the embedding associated with $f \in \Lambda^{\alpha}$.
There is a natural map
$\Phi : \left[ {\textstyle{\bigoplus_n}} R(n) \right] \!
 ^{(\Lambda^{\alpha})} \longrightarrow Y^{\alpha}$,
such that $\Phi \circ \nu_f = f$ for any
$f \in \Lambda^{\alpha}$. We now form the pushout of $J$
and $\Phi$, that defines $Y^{\alpha +1}$, $\iota_{\alpha}$
and $\Psi$ as follows
\[ \begin{array}{ccccccccc}
    0 & \longrightarrow
      & \left[ {\textstyle{\bigoplus_n}}
        R(n) \right] \! ^{(\Lambda^{\alpha})}
      & \stackrel{J}{\longrightarrow}
      & \left[ {\textstyle{\bigoplus_n}}
        R(n) \right] \! ^{(\Lambda^{\alpha})}
      & \longrightarrow
      & R(\infty) ^{(\Lambda^{\alpha})}
      & \longrightarrow & 0 \\
    & & \!\!\! {\scriptstyle{\Phi}} \downarrow
    & & \!\!\! {\scriptstyle{\Psi}} \downarrow
    & & \parallel & & \\
    0 & \longrightarrow & Y^{\alpha}
      & \stackrel{\iota_{\alpha}}{\longrightarrow}
      & Y^{\alpha +1} & \longrightarrow
      & R(\infty) ^{(\Lambda^{\alpha})}
      & \longrightarrow & 0
    \end{array} \]
We note that for any morphism
$f : \bigoplus_nR(n) \longrightarrow Y^{\alpha}$ we have
\[ \iota_{\alpha} \circ f =
   \iota_{\alpha} \circ \Phi \circ \nu_f =
   \Psi \circ J \circ \nu_f =
   \Psi \circ \nu_f \circ (1-I) . \]
Therefore, we obtain the commutative diagram (4), by letting
$g = \Psi \circ \nu_f$.

Having constructed the family $(Y^{\alpha})_{\alpha \leq \tau}$,
we define $Y = Y^{\tau} = \bigcup_{\alpha < \tau} Y^{\alpha}$
and let $\jmath_{\alpha} : Y^{\alpha} \longrightarrow Y$ denote
the inclusion for all $\alpha < \tau$. If
$Y^{\alpha} = (Y^{\alpha}_n)_n$ for all $\alpha < \tau$ and
$Y = (Y_n)_n$, then $Y_n$ is the ascending union of its submodules
$(Y^{\alpha}_n)_{\alpha < \tau}$ for all $n \geq 0$. In order to
show that $\Rlimitn\ Y_n = 0$, it suffices (in view of Corollary
1.2) to show that $\mbox{Ext}^1_{\mathcal C}(R(\infty),Y)=0$. In
other words, it suffices to show that any morphism
$h : \bigoplus_nR(n) \longrightarrow Y$ factors through $1-I$.
Since
\begin{equation}
 \mbox{Hom}_{\mathcal C} \! \left( {\textstyle{\bigoplus_n}}
 R(n),Y \right) = {\displaystyle{\prod}}_n
 \mbox{Hom}_{\mathcal C}(R(n),Y) \simeq
 {\displaystyle{\prod}}_n \mbox{Hom}_R(R,Y_n) =
 {\displaystyle{\prod}}_n Y_n , 
\end{equation}
the morphism $h$ corresponds to a sequence 
$(y_n)_n \in \prod_nY_n$. For each $n$, there exists an ordinal
number $\alpha(n) < \tau$, such that
$y_n \in Y^{\alpha(n)}_n \subseteq Y_n$. Since $\tau = \aleph_1$
and the supremum of a sequence of countable ordinals is also
countable, there exists an ordinal $\alpha < \tau$, such that
$y_n \in Y^{\alpha}_n \subseteq Y_n$ for all $n$; then,
$(y_n)_n \in \prod_nY^{\alpha}_n \subseteq \prod_nY_n$. The 
naturality of the isomorphism (5) with respect to the inverse 
sequence $Y$ implies that the following diagram is commutative
\[ \begin{array}{ccc}
   \mbox{Hom}_{\mathcal C} \! \left( {\textstyle{\bigoplus_n}}
   R(n),Y^{\alpha} \right) & \stackrel{\sim}{\longrightarrow} & 
   {\displaystyle{\prod}}_n Y^{\alpha}_n \\
   {\scriptstyle{(\jmath_{\alpha})_*}} \downarrow & & \downarrow \\
   \mbox{Hom}_{\mathcal C} \! \left( {\textstyle{\bigoplus_n}}
   R(n),Y \right) & \stackrel{\sim}{\longrightarrow} & 
   {\displaystyle{\prod}}_n Y_n 
   \end{array} \]
Here, the unlabelled vertical arrow is the inclusion. We conclude 
that $h$ factors through $Y^{\alpha}$ as the composition 
$\jmath_{\alpha} \circ f$, for a suitable morphism 
$f : \bigoplus_nR(n) \longrightarrow Y^{\alpha}$. If 
$g : \bigoplus_nR(n) \longrightarrow Y^{\alpha +1}$ is a morphism
making the diagram (4) above commutative, then
\[ h = \jmath_{\alpha} \circ f
     = \jmath_{\alpha +1} \circ \imath_{\alpha} \circ f
     = \jmath_{\alpha +1} \circ g \circ (1-I) \]
and hence $h$ factors through $1-I$, as needed.

Finally, we note that the quotient inverse sequence $Z=Y/X$ admits
a continuous ascending filtration by its inverse subsequences
$(Y^{\alpha}/X)_{\alpha \leq \tau} =
 (Y^{\alpha}/Y^0)_{\alpha \leq \tau}$.
The successive quotients
$(Y^{\alpha +1}/Y^0)/(Y^{\alpha}/Y^0) \simeq Y^{\alpha +1}/Y^{\alpha}$
are contained in $\mbox{Filt}(R(\infty))$ for all $\alpha < \tau$
and hence $Z$ is contained in
$\mbox{Filt}(\mbox{Filt}(R(\infty))) = \mbox{Filt}(R(\infty))$.
Invoking Lemma 2.1, it follows that $Z \simeq F(\infty)$ for a
free module $F$. \hfill $\Box$

\vspace{0.1in}

\noindent
{\bf Remarks 2.3.}
(i) Let $X=(X_n)_n$ be an inverse sequence and assume that
$\mbox{card} \, X_n \leq \kappa$ for all $n$, where $\kappa$
is an infinite cardinal with $\kappa \geq \mbox{card} \, R$.
A careful examination of the proof of Theorem 2.2 shows that the
inverse sequence $Y=(Y_n)_n$ therein is such that
$\mbox{card} \, Y_n \leq \kappa^{\aleph_0}$ for all $n$. In 
particular, the rank of the free module $F$ is bounded by 
$\kappa^{\aleph_0}$.

(ii) The short exact sequence of inverse sequences in the statement
of Theorem 2.2 is split in each degree and hence we obtain a 
(non-canonical) identification $Y_n = X_n \oplus F$ for all $n$. 
Then, the structure map $Y_n \longrightarrow Y_{n-1}$ of $Y$ is 
identified with
\[ \left( \begin{array}{cc} s_n & t_{n-1} \\ 0 & 1 \end{array}
   \right) : X_n \oplus F \longrightarrow X_{n-1} \oplus F , \]
where $s_n : X_n \longrightarrow X_{n-1}$ is the structure map
of $X$ and $t_{n-1} : F \longrightarrow X_{n-1}$ is some linear 
map.

(iii) If the short exact sequence of inverse sequences in the 
statement of Theorem 2.2 is split (as a short exact sequence 
of inverse sequences), then $\Rlimitn\ X_n=0$. Conversely, if 
$\Rlimitn\ X_n=0$ and $F$ is a free module with basis $\Lambda$,
then we may use Proposition 1.1(ii) and conclude that 
\[ \mbox{Ext}^1_{\mathcal C}(F(\infty),X) = 
   \Rlimitn\ \mbox{Hom}_R(F,X_n) \simeq 
   \Rlimitn\ \! \left( X_n^{\Lambda} \right) = 
   \left( \Rlimitn\ X_n \right) \! ^{\Lambda} = 0 . \]
Hence, the short exact sequence in the statement of Theorem 
2.2 splits (in ${\mathcal C}$). In that case, and only in that 
case, we may choose the identifications $Y_n = X_n \oplus F$ 
in (ii) above, so that the structure maps 
$Y_n \longrightarrow Y_{n-1}$ are given by the 
diagonal matrix
\[ \left( \begin{array}{cc} s_n & 0 \\ 0 & 1 \end{array}
   \right) : X_n \oplus F \longrightarrow X_{n-1} \oplus F , \]
where $s_n : X_n \longrightarrow X_{n-1}$ is the structure map
of $X$.

(iv) Assume that we have chosen the identifications 
$Y_n = X_n \oplus F$ for all $n$, so that the structure maps 
$Y_n \longrightarrow Y_{n-1}$ of $Y$ are given by upper 
triangular matrices as in (ii) above, involving linear maps 
$(t_{n-1} : F \longrightarrow X_{n-1})_n$. If $\Rlimitn\ X_n=0$,
then $\Rlimitn\ \mbox{Hom}_R(F,X_n) = 0$ as well; cf.\ (iii) 
above. Therefore, there exists a sequence of linear maps 
$(\tau_{n-1} : F \longrightarrow X_{n-1})_n$, so that 
$t_{n-1} = \tau_{n-1}-s_n \circ \tau_n$ for all $n$. Then, the 
map $F \longrightarrow X_n \oplus F = Y_n$, which is given by 
letting $z \mapsto (\tau_n(z),z)$, $z \in F$, is easily seen 
to be the degree $n$ component of a splitting (in ${\mathcal C}$) 
of the epimorphism $Y \longrightarrow F(\infty)$ in Theorem 2.2. 
This reaffirms that the short exact sequence in Theorem 2.2 is 
indeed split in this case, as we noted in (iii) above.
\addtocounter{Lemma}{1}

\section{The existence of special precovers}

\noindent
The notion of a cotorsion pair was introduced (for classes of
abelian groups) by Salce in \cite{S}. The importance of this
notion in the study of approximations of modules is analyzed
by G\"{o}bel and Trlifaj in \cite{GT}. In this section, we use
a trick by Salce that shows the equivalence between the existence
of enough injectives and the existence of enough projectives for
a cotorsion pair in a module category (see [loc.cit.] for the
definitions), to reformulate Theorem 2.2 in a form that will
directly imply the Theorem stated in the Introduction.

\begin{Theorem}
For any inverse sequence $X=(X_n)_n$ there is a functorial
short exact sequence of inverse sequences
\[ 0 \longrightarrow (Y_n)_n \longrightarrow (A_n)_n
     \longrightarrow (X_n)_n \longrightarrow 0 , \]
where $\Rlimitn\ Y_n=0$, the inverse sequence $(A_n)_n$
consists of free modules and its structure maps
$A_n \longrightarrow A_{n-1}$ are (split) monomorphisms
with a free cokernel in each degree.
\end{Theorem}
\vspace{-0.05in}
\noindent
{\em Proof.}
For any $n \geq 0$, we choose a free module $P_n$ together
with an epimorphism $P_n \longrightarrow X_n$. This choice
can be made functorially, by letting $P_n$ be the free module
on the underlying set of $X_n$. There is an induced morphism 
of inverse sequences $\pi_n : P_n(n) \longrightarrow X$, which
is surjective in degree $n$. The $\pi_n$'s induce a morphism
$\pi : \bigoplus_nP_n(n) \longrightarrow X$, which is surjective
in each degree; hence, $\pi$ is an epimorphism of inverse 
sequences. Let 
$Q = \bigoplus_nP_n(n)$, $K = \mbox{ker} \, \pi$ and consider
the short exact sequence of inverse sequences
\[ 0 \longrightarrow K \longrightarrow Q
     \stackrel{\pi}{\longrightarrow} X \longrightarrow 0 . \]
If $Q = (Q_n)_n$, then $Q_n = \bigoplus_{i=n}^{\infty}P_i$ and
the structure map $Q_n \longrightarrow Q_{n-1}$ is the obvious
(split) monomorphism with cokernel $P_{n-1}$ for all $n$.

We now apply Theorem 2.2 to the inverse sequence $K$ and obtain
a functorial short exact sequence of inverse sequences
\[ 0 \longrightarrow K \longrightarrow Y
     \longrightarrow F(\infty) \longrightarrow 0 , \]
where $Y = (Y_n)_n$ is such that $\Rlimitn\ Y_n=0$ and $F$ is
a free module. We consider the pushout diagram of the two
monomorphisms $K \longrightarrow Q$ and $K \longrightarrow Y$
in ${\mathcal C}$, defining the inverse sequence $A = (A_n)_n$
as follows
\begin{equation}
\begin{array}{ccccccccc}
    & & 0 & & 0 & & & & \\
    & & \downarrow & & \downarrow & & & & \\
    0 & \longrightarrow & K & \longrightarrow
    & Q & \stackrel{\pi}{\longrightarrow} & X
      & \longrightarrow & 0 \\
    & & \downarrow & & \downarrow & & \parallel & & \\
    0 & \longrightarrow & Y & \longrightarrow
    & A & \longrightarrow & X & \longrightarrow
    & 0 \\
    & & \downarrow & & \downarrow & & & & \\
    & & F(\infty) & = & F(\infty) & & & & \\
    & & \downarrow & & \downarrow & & & & \\
    & & 0 & & 0 & & & &
\end{array} 
\end{equation}
The vertical short exact sequence in the middle implies that
$A_n \simeq Q_n \oplus F$ is a free module for all $n$. An
application of the snake lemma to the following diagram
\[ \begin{array}{ccccccccc}
    0 & \longrightarrow & Q_n & \longrightarrow
    & A_n & \longrightarrow & F & \longrightarrow & 0 \\
    & & \downarrow & & \downarrow & & \parallel & & \\
    0 & \longrightarrow & Q_{n-1} & \longrightarrow
    & A_{n-1} & \longrightarrow & F & \longrightarrow & 0 \\
   \end{array} \]
shows that the structure map $A_n \longrightarrow A_{n-1}$ 
is a monomorphism, whose cokernel coincides with the cokernel 
of the structure map $Q_n \longrightarrow Q_{n-1}$, i.e.\ 
with $P_{n-1}$, for all $n$. In particular, the structure maps 
of the inverse sequence $A$ are (split) monomorphisms with a 
free cokernel in all degrees. Hence, the horizontal short 
exact sequence in the middle of diagram (6) satisfies all 
requirements in the statement. \hfill $\Box$

\vspace{0.1in}

\noindent
{\bf Remarks 3.2.}
(i) Let $X=(X_n)_n$ be an inverse sequence and assume that
$\mbox{card} \, X_n \leq \kappa$ for all $n$, where $\kappa$
is an infinite cardinal with $\kappa \geq \mbox{card} \, R$.
A careful examination of the proof of Theorem 3.1 shows that
the inverse sequence $K=(K_n)_n$ therein is such that
$\mbox{card} \, K_n \leq \kappa$ for all $n$. It follows from 
Remark 2.3(i) that the inverse sequence $Y=(Y_n)_n$ in Theorem 
3.1 is such that $\mbox{card} \, Y_n \leq \kappa^{\aleph_0}$ 
for all $n$ and the free module $F$ therein has rank 
$\leq \kappa^{\aleph_0}$. Therefore, the inverse sequence 
$A=(A_n)$ consists of free modules of rank 
$\leq \kappa^{\aleph_0}$ in all degrees.\footnote{In particular, 
if $R={\mathbb Z}$ and $X$ is an inverse sequence of countable 
abelian groups, then the inverse sequences $Y=(Y_n)_n$ and 
$A=(A_n)_n$ in Theorem 3.1 consist of abelian groups of 
cardinality $\leq 2^{\aleph_0}$.}

(ii) We adopt the notation in the proof of Theorem 3.1. If we
identify $A_n$ with the direct sum $Q_n \oplus F$ for all $n$,
then the structure map
$A_n \longrightarrow A_{n-1}$ of $A$ is identified with
\[ \left( \begin{array}{cc} \jmath_n & l_{n-1} \\ 0 & 1 
   \end{array} \right) \! : Q_n \oplus F \longrightarrow 
   Q_{n-1} \oplus F , \]
where $\jmath_n : Q_n \longrightarrow Q_{n-1}$ is the structure
map of $(Q_n)_n$ (a split monomorphism with cokernel $P_{n-1}$)
and $l_{n-1} : F \longrightarrow Q_{n-1}$ is some linear map.
The filtration of $Q_0 = \bigoplus_iP_i$ by the $Q_n$'s induces 
a linear topology therein, which is Hausdorff and the completion 
$\widehat{Q_0}$ is identified with the direct product $\prod_iP_i$, 
so that 
$\Rlimitn\ Q_n = \widehat{Q_0}/Q_0 = \prod_iP_i/\bigoplus_iP_i$.
Therefore, the 6-term lim-lim$^1$ exact sequence that is associated 
with the vertical short exact sequence in the middle of diagram (6) 
reduces to an exact sequence
\begin{equation}
 0 \longrightarrow \Limitn\ A_n \longrightarrow F 
   \stackrel{\delta}{\longrightarrow} \widehat{Q_0}/Q_0
   \longrightarrow \Rlimitn\ A_n \longrightarrow 0 . 
\end{equation}
For any $z \in F$ we consider the sequence 
$(l_n(z))_n$ in $Q_0$. Since $l_n(z) \in Q_n$ for all $n$, 
the infinite series $\sum_nl_n(z)$ converges in the completion 
$\widehat{Q_0}$. We leave it to the reader to verify that the 
map $\delta$ in (7) above maps $z \in F$ onto the class of 
$\sum_nl_n(z) \in \widehat{Q_0}$ in the quotient $\widehat{Q_0}/Q_0$.
\addtocounter{Lemma}{1}

\vspace{0.1in}

\noindent
We can now prove the Theorem stated in the Introduction.

\vspace{0.1in}

\noindent
{\em Proof.}
In view of Theorem 3.1, the inverse sequence $(X_n)_n$
fits into a short exact sequence
\[ 0 \longrightarrow (Y_n)_n \longrightarrow (A_n)_n
     \stackrel{p}{\longrightarrow} (X_n)_n
     \longrightarrow 0 , \]
where the inverse sequences $(Y_n)_n$ and $(A_n)_n$ enjoy
the properties stated therein. Then, $(A_n)_n$ is essentially
a filtration of the free module $A_0$ by free submodules,
such that the successive quotients $A_n/A_{n+1}$ are free
for all $n$. Moreover, the 6-term lim-lim$^1$ exact sequence
and the triviality of $\Rlimitn\ Y_n$ imply that $p$ induces
an isomorphism
$\Rlimitn\ A_n \stackrel{\sim}{\longrightarrow} \Rlimitn\ X_n$.
\hfill $\Box$

\vspace{0.1in}

\noindent
The consequence of Theorem 3.1 that is presented in Proposition
3.4 below was communicated to me by George Raptis.

\begin{Lemma}
Let $X=(X_n)_n$ be an inverse sequence and consider the inverse 
sequences $Y=(Y_n)_n$ and $A=(A_n)_n$, as well as the short exact 
sequence of inverse sequences 
\[ 0 \longrightarrow Y \stackrel{\nu}{\longrightarrow} A 
     \longrightarrow X \longrightarrow 0 \]
of Theorem 3.1. Then:

(i) The canonical map $\lim A \longrightarrow A_n$ is injective 
for all $n$.
  
(ii) The canonical map $\lim Y \longrightarrow Y_n$ is injective 
for all $n$.

(iii) If $W = (W_n)_n$ is the inverse sequence which is defined 
by letting $W_n$ be the cokernel of the canonical map 
$\lim Y \longrightarrow Y_n$ for all $n$, then 
$\Limitn\ W_n = \Rlimitn\ W_n = 0$.\footnote{In other words, 
the inverse sequence $W$ is {\em local}, in the sense of 
\cite[$\S $2]{CI}.}
\end{Lemma}
\vspace{-0.05in}
\noindent
{\em Proof.}
(i) This follows since the structure maps of the inverse sequence 
$A$ are injective.

(ii) This follows from (i), taking into account the commutativity
of the following diagram
\[ \begin{array}{ccc}
    \lim Y & \stackrel{\lim \nu}{\longrightarrow} & \lim A \\
    \downarrow & & \downarrow \\
    Y_n & \stackrel{\nu_n}{\longrightarrow} & A_n 
   \end{array} \]
Here, the vertical maps are the canonical maps from the limits 
and the horizontal map at the top is the injective map induced 
by the monomorphism $\nu = (\nu_n : Y_n \longrightarrow A_n)_n$.

(iii) Let us denote the limit $\lim Y$ by $L$. The canonical maps 
$(L \longrightarrow Y_n)_n$ define a morphism of inverse sequences
$L(\infty) \longrightarrow Y$, which induces an isomorphism on 
limits. Invoking (ii), we may consider the short exact sequence 
of inverse sequences
\begin{equation}
 0 \longrightarrow L(\infty) \longrightarrow Y 
   \longrightarrow W \longrightarrow 0 . 
\end{equation}
Since $\Rlimit L(\infty) = \Rlimit Y=0$, the 6-term lim-lim$^1$ 
exact sequence associated with this short exact sequence shows 
that both groups $\lim W$ and $\Rlimit W$ are trivial, as needed. 
\hfill $\Box$

\begin{Proposition}
For any inverse sequence $(X_n)_n$ there is an another inverse 
sequence $(B_n)_n$ and a morphism of inverse sequences
$g : (B_n)_n \longrightarrow (X_n)_n$, which depends functorially 
on $(X_n)_n$, such that:

(i) the inverse sequence $(B_n)_n$ represents a decreasing 
filtration of $B=B_0$ by submodules,

(ii) the successive quotients $B_n/B_{n+1}$ are free modules for all $n$ 
and 

(iii) $g$ induces isomorphisms 
$\Limitn\ B_n \stackrel{\sim}{\longrightarrow} \Limitn\ X_n$ and
$\Rlimitn\ B_n \stackrel{\sim}{\longrightarrow} \Rlimitn\ X_n$.
\end{Proposition}
\vspace{-0.05in}
\noindent
{\em Proof.}
We consider the inverse sequences $Y=(Y_n)_n$, $A=(A_n)_n$
and the short exact sequence
\[ 0 \longrightarrow Y \longrightarrow A 
     \longrightarrow X \longrightarrow 0 \]
of Theorem 3.1, which is functorial in $X=(X_n)_n$. We denote
the limit $\lim Y$ by $L$, as in the proof of Lemma 3.3(iii),
and consider the short exact sequence (8). Then, the pushout 
of the two morphisms $Y \longrightarrow W$ and 
$Y \longrightarrow A$ in ${\mathcal C}$ defines the inverse 
sequence $B = (B_n)_n$ and the morphism $g$, as pictured below
\begin{equation}
\begin{array}{ccccccccc}
    & & 0 & & 0 & & & & \\
    & & \downarrow & & \downarrow & & & & \\
    & & L(\infty) & = & L(\infty) & & & & \\
    & & \downarrow & & \downarrow & & & & \\
    0 & \longrightarrow & Y & \longrightarrow
    & A & \longrightarrow & X & \longrightarrow & 0 \\
    & & \downarrow & & \downarrow & & \parallel & & \\
    0 & \longrightarrow & W & \longrightarrow
    & B & \stackrel{g}{\longrightarrow} & X 
    & \longrightarrow & 0 \\
    & & \downarrow & & \downarrow & & & & \\
    & & 0 & & 0 & & & &
\end{array} 
\end{equation}
An
application of the snake lemma to the following diagram
\[ \begin{array}{ccccccccc}
    0 & \longrightarrow & L & \longrightarrow
    & A_n & \longrightarrow & B_n & \longrightarrow & 0 \\
    & & \parallel & & \downarrow & & \downarrow & & \\
    0 & \longrightarrow & L & \longrightarrow
    & A_{n-1} & \longrightarrow & B_{n-1} & \longrightarrow & 0 \\
   \end{array} \]
shows that the structure map $B_n \longrightarrow B_{n-1}$ 
is a monomorphism, whose cokernel coincides with the cokernel 
of the structure map $A_n \longrightarrow A_{n-1}$. In 
particular, the structure maps of the inverse sequence $B$ 
are (split) monomorphisms with a free cokernel in all degrees. 
Since $\Limitn\ \, W_n \! = \Rlimitn\ W_n \! = 0$ (cf.\ Lemma 
3.3(iii)), the 6-term lim-lim$^1$ exact sequence associated with 
the horizontal short exact sequence at the bottom of diagram (9) 
implies that the morphism $g$ induces isomorphisms
$\Limitn\ B_n \stackrel{\sim}{\longrightarrow} \Limitn\ X_n$ and
$\Rlimitn\ B_n \stackrel{\sim}{\longrightarrow} \Rlimitn\ X_n$. 
\hfill $\Box$

\bigskip

\noindent
{\em Acknowledgments.}
It is a pleasure to thank George Raptis for useful comments and 
suggestions on an earlier draft of this paper, that improved its 
quality.

\vspace{0.05in}

{\small {\sc Department of Mathematics,
             University of Athens,
             Athens 15784,
             Greece}}

{\em E-mail address:} {\tt emmanoui@math.uoa.gr}

\end{document}